\documentclass[11 pt]{amsart}
\usepackage[english]{babel}
\usepackage{amsthm}
\usepackage{graphicx}
\usepackage{amsfonts}
\usepackage{eufrak}
\usepackage{amsopn}
\usepackage{amsmath, amsfonts, amssymb, amscd, amsthm}
\usepackage{yhmath}
\usepackage{setspace}
\usepackage{mathtools}
\usepackage[all]{xy}
\newtheorem{theorem}{Theorem}[section]

\newtheorem{lemma}[theorem]{Lemma}
\newtheorem{proposition}[theorem]{Proposition}
\newtheorem{corollary}[theorem]{Corollary}
\theoremstyle{definition}

\theoremstyle{remark}
\newtheorem{remark}[theorem]{Remark}
\theoremstyle{example}

\theoremstyle{note}

\numberwithin{equation}{section}

\DeclareMathOperator{\Sp}{Sp}

\DeclareMathOperator{\Mod}{Mod}

\begin{document}
\title{Finite order elements in the integral symplectic group}
\author{Kumar Balasubramanian} \thanks{Research of Kumar Balasubramanian was supported by DST-SERB Grant: YSS/2014/000806.}

\address{Kumar Balasubramanian\\
Department of Mathematics\\
IISER Bhopal\\
Bhopal, Madhya Pradesh 462066, India}
\email{bkumar@iiserb.ac.in}

\author{M. Ram Murty} \thanks{Research of M. Ram Murty was partially supported by an NSERC Discovery grant.}
\address{M. Ram Murty\\
Department of Mathematics and Statistics\\
Queen's University\\
Kingston, Ontario K7L 3N6, Canada}
\email{murty@mast.queensu.ca}

\author{Karam Deo Shankhadhar}
\address{Karam Deo Shankhadhar\\
Department of Mathematics\\
IISER Bhopal\\
Bhopal, Madhya Pradesh 462066, India}
\email{karamdeo@iiserb.ac.in}

\thispagestyle{empty}

\maketitle
\vspace{0.3 cm}

\section*{Abstract}
For $g\in \mathbb{N}$, let $G=\Sp(2g,\mathbb{Z})$ be the integral symplectic group and $S(g)$ be the set of all positive integers which can occur as the order of an element in $G$. In this paper, we show that $S(g)$ is a bounded subset of $\mathbb{R}$ for all positive integers $g$. We also study the growth of the functions $f(g)=|S(g)|$, and $h(g)=max\{m\in \mathbb{N}\mid m\in S(g)\}$ and show that they have at least exponential growth.

\section{Introduction}

Given a group $G$ and a positive integer $m\in \mathbb{N}$, it is natural to ask if there exists $k\neq 1\in G$ such that $o(k)=m$, where $o(k)$ denotes the order of the element $k\in G$. In this paper, we make some observations about the collection of positive integers which can occur as orders of elements in $G=\Sp(2g, \mathbb{Z})$. Before we proceed further we set up some notation and briefly mention the problems studied in this paper. \\

Let $G=\Sp(2g,\mathbb{Z})$ be the group of all $2g \times 2g$ matrices with integral entries satisfying
\[A^{\top}JA=J\]
where $A^{\top}$ is the transpose of the matrix $A$ and $J=\begin{pmatrix}0_{g} & I_{g} \\ -I_{g} & 0_{g}\end{pmatrix}$.\\

Throughout we write $m=p_{1}^{\alpha_{1}}\dots p_{k}^{\alpha_{k}}$, where $p_{i}$ is a prime and $\alpha_{i}>0$ for all $i\in \{1,2, \ldots, k\}$. We also assume that the primes $p_{i}$ are such that $p_{i}<p_{i+1}$ for $1\leq i < k$. We write $\pi(x)$ for the number of primes less than or equal to $x$. Also for $A\in G$ we let $o(A)$ denote the order of $A$. We let $\phi$ denote the Euler's phi function. It is a well known fact that the function $\phi$ is multiplicative, i.e., $\phi(mn)=\phi(m)\phi(n)$ if $m,n$ are relatively prime and satisfies $\phi(p^{\alpha})= p^{\alpha}(1-\frac{1}{p})$ for all primes $p$ and positive integer $\alpha\in \mathbb{N}$ (see \cite{Burton} for a proof). Let \[S(g)=\{m\in \mathbb{N}\mid \exists\, A\neq 1 \in G \textrm{ with } o(A)=m\}.\]
In this paper we show that $S(g)$ is always a bounded subset of $\mathbb{R}$ for all positive integers $g$. Once we know that $S(g)$ is a bounded set, it makes sense to consider the functions $f(g)=|S(g)|$, where $|S(g)|$ is the cardinality of $S(g)$ and $\displaystyle h(g)= max \{m\mid m\in S(g)\}$, i.e., $h(g)$ is the maximal possible (finite) order in $G=\Sp(2g,\mathbb{Z})$. We show that the functions $f$ and $h$ have at least exponential growth.\\

The above problem derives its motivation from analogous problems from the theory of mapping class groups of a surface of genus $g$. We know that given a surface $S_{g}$ of genus $g$, there is a surjective homomorphism $\psi: \Mod(S_{g})\rightarrow \Sp(2g,\mathbb{Z})$, where $\Mod(S_{g})$ is the mapping class group of $S_{g}$. It is a well known fact that for $f\in \Mod(S_{g})$ ($f\neq 1$) of finite order, we have $\psi(f)\neq 1$. Let $\tilde{S}(g)=\{m\in \mathbb{N}\mid \exists f\neq 1\in \Mod(S_{g}) \textrm{ with } o(f)=m\}$. The set $\tilde{S}(g)$ is a finite set and it makes sense to consider the functions $\tilde{f}(g)=|\tilde{S}(g)|$ and $\tilde{h}(g)=max\{m\in \mathbb{N}\mid m\in \tilde{S}(g)\}$. It is a well known fact that both these functions $\tilde{f}$ and $\tilde{h}$ are bounded above by $4g+2$. We refer the reader to \cite{FarMar} for an excellent introduction to the mapping class group and the proofs of some of these facts.

\section{Some results we need}
In this section we mention a few results that we need in order to prove the main results in this paper.

\begin{proposition}[B\"{u}rgisser]\label{burgisser} Let $\displaystyle m= p_{1}^{\alpha_{1}}\dots p_{k}^{\alpha_{k}}$, where the primes $p_{i}$ satisfy $p_{i}<p_{i+1}$ for $1\leq i < k$ and where $\alpha_{i}\geq 1$ for $1\leq i \leq k$. There exists a matrix $A\in \Sp(2g, \mathbb{Z})$ of order $m$ if and only if\\

\begin{enumerate}
\item[a)] $\displaystyle \sum_{i=2}^{k}\phi(p_{i}^{\alpha_{i}})\leq 2g$, if $m\equiv 2(\bmod\, 4)$.
\item[b)] $\displaystyle \sum_{i=1}^{k}\phi(p_{i}^{\alpha_{i}})\leq 2g$, if $m\not \equiv 2(\bmod\, 4)$.
\end{enumerate}
\end{proposition}
\vspace{0.2 cm}

\begin{proof}
See corollary 2 in \cite{Bur} for a proof.
\end{proof}
\vspace{0.2 cm}

\begin{proposition}[Dusart]\label{Dusart} Let $p_{1}, p_{2},\dots, p_{n}$ be the first $n$ primes. For $n\geq 9$, we have\\
\[p_{1}+p_{2}+\cdots+ p_{n}< \frac{1}{2}np_{n}.\]
\end{proposition}
\vspace{0.2 cm}

\begin{proof} See theorem 1.14 in \cite{Dus} for a proof.
\end{proof}
\vspace{0.2 cm}

%

\begin{proposition}[Dusart]\label{Dusart2} For $x\geq 2$, $\pi(x)\leq \frac{x}{\log x}\bigg (1 + \frac{1.2762}{\log x}\bigg)$. For $x\geq 599$, $\pi(x)\geq \frac{x}{\log x}\bigg (1 + \frac{1}{\log x}\bigg)$.
\end{proposition}

\begin{proof} See theorem 6.9 in \cite{Dus[1]} for a proof. 
\end{proof}

\begin{proposition}[Dusart]\label{Dusart1} For $x\geq 2973$, \\
\[\prod_{p\leq x}\bigg (1-\frac{1}{p}\bigg) > \frac{e^{-\gamma}}{\log x}\bigg(1-\frac{0.2}{(\log x)^{2}}\bigg).\]
\vspace{0.2 cm}\\
where $\gamma$ is the Euler's constant.\\
\end{proposition}

\begin{proof}
See theorem 6.12 in \cite{Dus[1]} for a proof.
\end{proof}

\begin{proposition}[Rosser]\label{Rosser} For $x\geq 55$, we have $\pi(x)> \frac{x}{\log x + 2}$.
\end{proposition}

\begin{proof} See theorem 29 in \cite{Ros} for a proof.
\end{proof}

\section{Main Results}
In this section we prove the main results of this paper. To be more precise, we prove the following.\\

\begin{enumerate}
\item[a)] $S(g)$ is a bounded subset of $\mathbb{R}$.
\item[b)] $f(g)=|S(g)|$ has at least exponential growth.
\item[c)] $h(g)=max\{m\mid m\in S(g)\}$ has at least exponential growth. \\
\end{enumerate}

\subsubsection{Boundedness of $S(g)$} For each $g\in \mathbb{N}$, let $S(g)=\{m\in \mathbb{N}\mid \exists\, A\neq 1 \in G \textrm{ with } o(A)=m\}$. In this section we show that $S(g)$ is a bounded subset of $\mathbb{R}$.\\

Let $m = p_1^{\alpha_1} \ldots p_k^{\alpha_k} \in S(g)$. Suppose $p_{i}>2g+1$ for some $i\in \{1, 2, \ldots, k\}$. This would imply that $\phi(p_{i}^{\alpha_{i}})=p_{i}^{\alpha_{i}-1}(p_{i}-1)>2g$, which contradicts proposition~\ref{burgisser}. It follows that all primes in the factorization of $m$ should be $\leq 2g+1$ and hence $k\leq g+1$.\\

\begin{theorem}\label{boundedness} For $g\in \mathbb{N}$, $S(g)$ is a bounded subset of $\mathbb{R}$.
\end{theorem}

\begin{proof} For $g\in \mathbb{N}$, fix $k=\pi(2g+1)$ and $P=\{p_{1}, p_{2}, \dots, p_{k}\}$ be the set of first $k$ primes arranged in increasing order.
The prime factorization of any $m \in S(g)$ involves primes only from the set $P$. The total number of non-empty subsets of $P$ is $2^{k}-1$. Let us denote the collection of these subsets of
$P$ as $\{P_1,P_2, \ldots P_{2^k-1}\}$. For $1\leq a \leq 2^{k}-1$, let
$P_{a}$ denote the subset $\{q_{1}, q_{2}, \dots, q_{n}\}$ of $P$, where $n=n(P_{a})$ is the number of primes in the subset $P_{a}$. For a fixed $a$ (and hence fixed $P_a$), define\\
\begin{align*}m_{a} &=m_{a}(\alpha_{1}, \dots, \alpha_{n}) = q_{1}^{\alpha_{1}}q_{2}^{\alpha_{2}}\dots q_{n}^{\alpha_{n}},\\\\
r_{a} &=r_{a}(\alpha_{1}, \dots, \alpha_{n}) = \displaystyle \sum_{i=1}^{n}q_{i}^{\alpha_{i}}\bigg(1-\frac{1}{q_{i}}\bigg),\end{align*}
where $\alpha_{i} > 0$. The key idea of the proof is to maximize the function $m_{a}$ considered as a function of the real variables $(\alpha_1,\alpha_2, \ldots \alpha_n)$ with respect to the inequality constraint $r_{a}\leq 2g+1$. We let $M_{a}$ denote this maximum. Using the Lagrange multiplier method we see that the function $m_{a}$ attains the maximum $M_a$ precisely when $q_{i}^{\alpha_{i}}(1-\frac{1}{q_{i}})= q_{j}^{\alpha_{j}}(1-\frac{1}{q_{j}})$ for all $1\leq i, j \leq n$. Under the above condition, the constraint
$r_a\leq 2g+1$ gives us $q_i^{\alpha_i} (1-\frac{1}{q_i}) \leq \frac{2g+1}{n}$, for any $1\leq i \leq n$.
Now

$$m_a(\alpha_1,\alpha_2,\ldots \alpha_n) = \frac{q_{1}^{\alpha_{1}}\big(1-\frac{1}{q_{1}}\big)q_{2}^{\alpha_{2}}\big(1-\frac{1}{q_{2}}\big)\ldots q_{n}^{\alpha_{n}}\big(1-\frac{1}{q_{n}}\big)}
{\displaystyle \prod_{i=1}^{n}\bigg(1-\frac{1}{q_{i}}\bigg)}.$$

From this it follows that for $1\leq a \leq 2^{k}-1$,

\[M_a = \frac{\bigg(q_1^{\alpha_1}\big(1-\frac{1}{q_{1}}\big)\bigg)^n}{{\displaystyle \prod_{i=1}^{n}\bigg(1-\frac{1}{q_{i}}\bigg)}} \leq \frac{\bigg(\frac{2g+1}{n}\bigg)^{n}}{\displaystyle \prod _{i=1}^{k}\bigg(1-\frac{1}{p_{i}}\bigg )}.\]
%

Therefore, for $m\in S(g)$, we have\\

\begin{align*}
m & \leq \max\limits_{1\leq a\leq 2^{k}-1} M_a\\
\vspace{0.1 cm}
& \leq  \frac{\max\limits_{1\leq a\leq 2^{k}-1}\bigg(\frac{2g+1}{n}\bigg)^{n}}{\displaystyle \prod _{i=1}^{k}\bigg(1-\frac{1}{p_{i}}\bigg )}\\
\vspace{0.1 cm}
& \leq \frac{e^{\frac{2g+1}{e}}}{\displaystyle \prod _{i=1}^{k}\bigg(1-\frac{1}{p_{i}}\bigg )}\\
\end{align*}

\noindent In the above computation, we have used the fact that for $x>0$, $\bigg (\frac{2g+1}{x}\bigg )^{x}$ attains the maximum when $x=(2g+1)/e$.\\

Observing that $\displaystyle\prod_{i=1}^{k} \bigg(1-\frac{1}{p_{i}}\bigg)\geq \frac{1}{2}\frac{2}{3}\bigg(\frac{4}{5}\bigg)^{\pi(2g+1)-2}$, we have

\[m\leq 3(5/4)^{\pi(2g+1)-2}e^{\frac{2g+1}{e}}\leq 3e^{\big(\frac{2g+1}{e}+ g-1\big)}\leq 3e^{3g}. \]

\end{proof}

\begin{corollary}\label{bound} For $g\in \mathbb{N}$, $f(g)\leq h(g)\leq 3e^{3g}$.
\end{corollary}

\begin{proof} For $m\in S(g)$, we have $m\leq 3e^{3g}$. The result follows.
\end{proof}
\vspace{0.2 cm}

\begin{remark} \textit{Upper bound for $S(g)$ for $g\geq 1486$}: The bound obtained in theorem~\ref{boundedness} is an absolute upper bound for $S(g)$. For $g\geq 1486$ , we can improve the above upper bound as  follows: Using proposition~\ref{Dusart1}, we get
\[\prod _{i=1}^{k}\bigg(1-\frac{1}{p_{i}}\bigg ) > \frac{1}{2}\frac{e^{-\gamma}}{\log(2g+1)}. \]
Therefore it follows that for $m\in S(g)$, we have
\[m\leq \frac{e^{\frac{2g+1}{e}}}{\displaystyle \prod _{i=1}^{k}\bigg(1-\frac{1}{p_{i}}\bigg )} \leq 2e^{\gamma}\log(2g+1)e^{\frac{2g+1}{e}}.\]
\end{remark}
\vspace{0.2 cm}

\subsubsection{Growth of $f(g)$ and $h(g)$}

In the previous section, we computed an upper bound for the functions $f(g)$ and $h(g)$. In this section we show that $f(g)$ and $h(g)$ have at least exponential growth. 

\begin{lemma}\label{dusart application} For $x\geq 23$, we have \[\displaystyle \sum_{p\leq x}p < \frac{1}{2}x\pi(x) \]
where the sum is over all primes $p\leq x$.
\end{lemma}

\begin{proof} Let $n$ be such that $p_{n}\leq x < p_{n+1}$, where $p_{n}$ denotes the $n^{th}$ prime number. It follows from proposition~\ref{Dusart}, that for $x\geq 23$, we have \[\displaystyle \sum_{p\leq x}p= \sum_{p\leq p_{n}}p < \frac{1}{2}np_{n}\leq \frac{1}{2}\pi(x)x.\]
\end{proof}

Before we proceed further, we set up some notation which we need in the following results. \\

\noindent Let $K (\geq e) \in \mathbb{N}$ be such that for $\sqrt{K\log K}\geq 23$.

\begin{lemma}\label{estimate for pi(x)} For $g\geq K$, $\pi(\sqrt{g\log(g)}) < \frac{3\sqrt{g\log(g)}}{\log(g\log(g))}.$
\end{lemma}

\begin{proof} For $y>1$, we have $\pi(y) < \frac{y}{\log(y)}\bigg(1+\frac{3}{2\log(y)}\bigg)$ (see proposition~\ref{Dusart2}). Using this estimate we get,

\begin{align*}
\pi(\sqrt{g\log(g)}) & < \frac{\sqrt{g\log(g)}}{\log(\sqrt{g\log(g)})}\bigg ( 1+ \frac{3}{2\log(\sqrt{g\log(g)})}\bigg )\\
& \\
& \leq \frac{\sqrt{g\log g}}{\log(\sqrt{g\log g})}\bigg(1 + \frac{3}{2\log 23}\bigg)\\
&\\
& = \frac{3\sqrt{g\log(g)}}{\log(g\log(g))}.
\end{align*}

\end{proof}

\begin{lemma}
Let $x=\sqrt{g\log(g)}$ and $\displaystyle m= m(g)=\prod_{p\leq x}p$. Then for $g\geq K$, we have $m\in S(g)$.
\end{lemma}

\begin{proof} By proposition~\ref{burgisser}, it is enough to show that $\beta= \displaystyle \sum_{2\neq p\leq x}(p-1)\leq 2g$. Using lemma~\ref{dusart application} and lemma~\ref{estimate for pi(x)} , we have\\

\begin{align*}\beta  < \sum_{p\leq x}p  & < \frac{1}{2}(\sqrt{g\log(g)})\pi(\sqrt{g\log(g)})\\
& < \frac{3}{2}\frac{g\log(g)}{\log(g\log(g))} = \frac{3}{2}g.\\
\end{align*}


\end{proof}

For $g\geq K$, let $A(g)=\{p\in \mathbb{N}\mid p\leq \sqrt{g\log(g)}\}$ and $m=m(g)$ be as above. If $d$ is any divisor of $m$, then it is easy to see that $d\in S(g)$. Also it is clear that the divisors $d$ of $m$ are in bijection with the number of subsets of $A(g)$. Since any divisor $d$ of $m$ is an element in $S(g)$ and the number of divisors correspond bijectively with subsets of $A(g)$, it follows that $f(g)=|S(g)|\geq  2^{\pi(\sqrt{g\log(g)})}$ (since number of subsets of $A(g)=2^{\pi(\sqrt{g\log(g)})}$). \\

We will now show that $|S(g)| > e^{\frac{1}{4}\sqrt{\frac{g}{\log(g)}}}$ from which it follows that the function $f(g)=|S(g)|$ has at least exponential growth. 

\begin{theorem} Let $L \in \mathbb{N}$ such that $\sqrt{L\log L}\geq 55$. Then $f(g)=|S(g)|> e^{\frac{1}{4}\sqrt{\frac{g}{\log(g)}}}$ for all $g\geq L$.
\end{theorem}

\begin{proof} From proposition~\ref{Rosser}, we have for all $g\geq L$, \\
\[\frac{\sqrt{g\log(g)}}{\log(g\log(g))} < \pi(\sqrt{g\log(g)}).\]
\vspace{0.2cm}\\

\noindent From this it follows that for all $g\geq L$, we have \\

\[f(g) \geq 2^{\pi(\sqrt{g\log(g)})} > 2^{\frac{\sqrt{g\log(g)}}{\log(g\log(g))}} > 2^{\frac{1}{2}\sqrt{\frac{g}{\log(g)}}} > e^{\frac{1}{4}\sqrt{\frac{g}{\log(g)}}}.
\]

\end{proof}

%
%
%
%

\begin{corollary} Let $L \in \mathbb{N}$ be as in the above theorem. Then $h(g) > e^{\frac{1}{4}\sqrt{\frac{g}{\log(g)}}}$ for all $g\geq L$.
\end{corollary}

\begin{proof}
Since $h(g)\geq f(g)$, the result follows.
\end{proof}

\begin{remark} For $g\log g\geq (599)^{2} $, we can improve the above lower bound $e^{\frac{1}{4}\sqrt{\frac{g}{\log g}}}$ to $e^{\sqrt{\frac{g}{4\log g}}}$ by using proposition~\ref{Dusart2}.
\end{remark}

\bibliographystyle{amsplain}
\bibliography{ref}
%
%
%

\end{document}